\documentclass{article}

\usepackage{amssymb}
\usepackage{amsthm}
\usepackage{amsmath,amssymb,amsopn,amsfonts,mathrsfs,amsbsy,amscd}
\usepackage{color}
\newcommand{\A}{\mathcal{A}}
\newcommand{\g}{\mathfrak{g}}
\newcommand{\h}{{\mathfrak{h}}}
\newcommand{\s}{{\mathfrak{s}}}
\newcommand{\ad}{{\mathrm{ad}}}
\newcommand{\M}{{\cal M}}
\newcommand{\p}{{\mathfrak{p} }}
\newtheorem{definition}{Definition}[section]
\newtheorem{theorem}{Theorem}[section]
\newtheorem{proposition}{Proposition}[section]
\newtheorem{lemma}{Lemma}[section]
\newtheorem{corollary}{Corollary}[section]

\newtheorem{remark}{Remark}[section]
\numberwithin{equation}{section}
\def\K{\mathbb{ K}}

\begin{document}
		\title{Non-existence of symmetric biderivations on finite-dimensional perfect Lie algebras}
		\author{
		Ignacio Bajo\footnote{Depto. Matem\'atica Aplicada II, E. I. Telecomunicaci\'on, Universidade de Vigo, 36310 Vigo, Spain. ibajo@dma.uvigo.es }, Sa\"{\i}d Benayadi\footnote{{\it Corresponding author.} Universit\'e de Lorraine, Laboratoire IECL, CNRS-UMR 7502, UFR MIM, 3 rue Augustin Frenel, BP 45112, 57073 Metz Cedex 03, France. said.benayadi@univ-lorraine.fr. }, Hassan Oubba\footnote{Universit\'e Moulay Isma\"{i}l,
			Facult\'e des sciences -  D\'epartement de Math\'ematiques,
			B.P. 11201, Zitoune, M\'ekn\`es 50000, Morocco. 
			hassan.oubba@edu.umi.ac.ma}}
	
	\maketitle

\begin{abstract}
	We show that there are no symmetric non-zero biderivations on perfect  Lie algebras of finite dimension over a field of characteristic zero. We show that this is equivalent to  show that every symmetric biderivation on a finite-dimensional perfect Lie algebra over such a field with values in a finite-dimensional module vanishes identically. This answers an open question posed by M. Bre\v{s}ar and K. Zhao in \cite{br2}.  
\end{abstract}
	%

\section{Introduction} \label{Introduction}
Let $(\g,[\,,\,])$ be a Lie algebra and $M$  a vector space. A {\it representation of $\g$ on $M$} is a morphism of Lie algebras $\psi: \g \to {\mathfrak{gl}}(M)$, this is to say, $\psi$ is a linear map such that $\psi([x,y])= 
\psi(x)\psi(y) - \psi(y)\psi(x)$ for all $x, y \in \g$. In this case $M$ is also called a {\it $\g$-module}.  
A linear map $D: \g \to M$ is called a {\it derivation of $\g$ with values in $M$} if
$D([x,y])= \psi(x)D(y) - \psi(y)D(x)$ holds for all $x,y \in \g.$ Notice that this means that
 $D$ is a $1$-cocycle  for the cohomology of $\g$ with values in $M$ (i.e.  $D\in Z^1(\g,M)$).
A {\it biderivation $\varphi$ of $\g$ with values in $M$} is a bilinear map $\varphi: \g\times \g \to M$  such that $\varphi(x,\cdot)$ and $\varphi(\cdot,x)$ are derivations of $\g$ with values in $M$, for all $x\in\g.$  If $\varphi$ is also symmetric, then $\varphi$ is called a {\it symmetric biderivation of $\g$ with values in $M$}.
If $M= \g$ and $\psi$  is  the adjoint representation of $\g$, this is to say, $\psi(x)(y)= \ad(x)y=[x,y],$ for all $x,y \in \g$, a biderivation of $\g$ with  values in $M= \g$ will be called just  a {\it biderivation of $\g$}.

Let $(\A,\cdot)$ be a nonassociative algebra and denote by $\mathrm{L}_x$ and $\mathrm{R}_x$, respectively, the left and right multiplications by an element $x\in\A$. Let us also denote  $\A^-= (A,[\,,\,]),$ where $[\,,\,]$ is the usual commutator of $(\A,\cdot)$. In \cite{BO} it is said that $(\A,\cdot)$ is an {\it algebra of biderivation type}, or an {\it $A_{BD}$-algebra} for short, if $\mathrm{L}_x$ and $\mathrm{R}_x$ are derivations of $\A^-$ for all $x\in\A$.   It is easily proved that the product in an $A_{BD}$-algebra $(\A,\cdot)$ is Lie-admissible and, therefore, $\A^-$ is a Lie algebra. Further, the  multiplication in $(\A,\cdot)$ provides a biderivation of the Lie algebra $\A^-.$  It was shown in \cite{BO} that this class of $A_{BD}$-algebras includes Lie algebras, symmetric Leibniz algebras, Lie-admissible left (or right) Leibniz algebras, Milnor algebras, and LR-algebras. Moreover, it has also been  proved that there is a one-to-one correspondence between the class of $A_{BD}$-algebras and the class of Lie algebras with symmetric biderivations. Explicitly, if $(\g,\varphi)$ is a Lie algebra with a symmetric biderivation, one defines on the vector space $\A=\g$ the new product $x\cdot y=\frac{1}{2}(\varphi(x,y)+[x,y])$, for $x,y\in\g$, and, conversely, given an $A_{BD}$-algebra $(\A,\cdot)$, one defines on $\g=\A^-$ the symmetric biderivation defined by $\varphi(x,y)=x\cdot y+y\cdot x$ for all $x,y\in\A^-$. In the sequel, if $(\g,[\,,\,])$ is a Lie algebra and $\varphi$ is a symmetric biderivation, we will frequently use the notation $\varphi(x,y)=x\circ y$ and we will say that $\varphi$ or $\circ$ is an {\it $A_{BD}$-structure on $\g$}. We will say that the symmetric biderivation is trivial if it is identically null.

A particular case of $A_{BD}$-structure on Lie algebras is provided by a so-called {\it commutative post-Lie algebra structure} (or {\it CPA-structure}) \cite{BD1,burde2}, which consist on a  symmetric  biderivation $\circ$ on the Lie algebra $\g$ satisfying the additional identity
	$[x,y]\circ  z= x \circ  (y  \circ z)- y \circ  (x \circ  z)$, for all $x,y,z\in\g$. Notice that this last identity is equivalent to the fact  that left multiplication with respect to  $\circ$  gives a representation of $\g$ on $\g$. It was proved in  \cite{burde2} that every $CPA$-structure on a real or complex perfect Lie algebra $\g$ (i.e. such that $[\g,\g]=\g$)  is trivial. Further, in \cite{BO} it is proved that every $A_{BD}$-structure on a semisimple real or complex Lie algebra must  vanish identically. Therefore the natural question of the existence of nontrivial $A_{BD}$-structures on perfect Lie algebras arises naturally.

Besides, in the paper \cite{br2} by  M. Bre\v{s}ar and K. Zhao it is posed the question of whether there exists a perfect Lie algebra $\g$ of finite dimension  and  a $\g$-module $M$ of finite dimension with $M^{\g}= \{ v\in M\ :\ \psi(\g)v=\{0\}\}=\{0\}$ admitting a nonzero symmetric biderivation $\varphi$ of $\g$ with values in $M$. Obviously, the existence of a nontrivial $A_{BD}$-structure on a finite-dimensional perfect Lie algebra, would answer Bre\v{s}ar and  Zhao's open question affirmatively. It should be remarked that in \cite{br2} an example of a nonzero symmetric biderivation with values in  an infinite-dimensional module of ${\mathfrak{sl}}(2)$ is given, and that such example has been used in \cite{BO} to find an infinite-dimensional perfect Lie algebra that can be endowed with a nonzero $A_{BD}$-structure. However, both problems in finite dimension remained unsolved until now.

The main goal of this paper is to answer Bre\v{s}ar and  Zhao's open question in the negative. Actually we will prove that, if $\varphi$ is a symmetric biderivation of a perfect finite-dimensional Lie algebra $\g$ with values in finite-dimensional $\g$-module $M$, then $\varphi=0$, even in the case $M^{\g}\ne\{0\}$.  The structure of the paper is as follows. We first show that Bre\v{s}ar and  Zhao's question is actually equivalent to the question of existence of nontrivial $A_{BD}$-structures on finite-dimensional perfect Lie algebras. Next, we show that every $A_{BD}$-structure on a Lie algebra with Levi decomposition $\g=\s\ltimes_\phi R(\g)$ is completely defined by three maps $F:\s\to R(\g)$, $G:\s\to R(\g)$ and $\Delta: R(\g)\to \mbox{Der}(R(\g))$ verifying certain equalities.  
In Section \ref{reduc}, we focus on perfect Lie algebras and prove that if there exist a  perfect Lie algebra with a nontrivial $A_{BD}$-structure then we can find a perfect Lie algebra with abelian radical also admitting nonzero $A_{BD}$-structures. This reduces our problem to the case in which the radical is commutative. The nonexistence of $A_{BD}$-structures in perfect Lie algebras with abelian radical is done in our last section, completing the proof of our non-existence result.

\section{An equivalence result}

 From now on, unless otherwise stated, all the vector spaces considered in this  paper  are finite-dimensional  vector spaces over a field $\K$ of characteristic zero.  

Let us recall some of the definitions given in the introduction.

\begin{definition} {\em Let $\psi: \g \to {\mathfrak{gl}}(M)$ be a representation of a Lie algebra $\g$ on a vector space $M$. A {\it biderivation $\varphi$ of $\g$ with values in $M$} is a bilinear map $\varphi: \g\times \g \to M$  such that 
\begin{equation}\label{phibider}\varphi(x,[y,z])=\psi(y)\varphi(x,z)-\psi(z)\varphi(x,y),\quad \varphi([y,z],x)=\psi(y)\varphi(z,x)-\psi(z)\varphi(y,x)\end{equation}
hold for all $x,y,z\in\g.$  If, further, $\varphi(x,y)=\varphi(y,x)$ for all $x,y\in\g$, then $\varphi$ is called a {\it symmetric biderivation of $\g$ with values in $M$.}}
\end{definition}

\begin{definition} {\em An {\it $A_{BD}$-structure} on a Lie algebra $\g$ is a commutative product $\circ:\g\times\g\to\g$ which verifies $x\circ [y,z]=[x\circ y,z]+[y,x\circ z]$ for any $x,y,z\in\g$. This is equivalent to the bilinear map $\varphi$ defined by  $\varphi(x,y)=x\circ y$ being a symmetric biderivation of  $\g$ with values on the adjoint module  $\g$.}
\end{definition}

\begin{theorem}\label{Equi}
There exists a nonzero symmetric biderivation of a perfect Lie algebra with values in a module if and only if there exists  a perfect Lie algebra admitting a nonzero $A_{BD}$-structure.
\end{theorem}

\begin{proof}
If there exists  a perfect Lie algebra $\g$ with a nonzero $A_{BD}$-structure $\circ$, then $\varphi(x,y)=x\circ y$ obviously defines a nonzero symmetric biderivation  of $\g$ with  values in the adjoint module $M=\g.$

Let us then prove the converse. Let $\psi: \p \to {\mathfrak{gl}}(M)$ be a representation of a perfect Lie algebra $(\p, [ , ])$ and $\varphi=\p\times \p \to M$  a symmetric biderivation with values in $M$. Define the sets $\M_0=\varphi(\p,\p)$ and $\M_k=\psi(\p)\M_{k-1}$ for all $k\ge 1$ and let $V$ be the linear span of $\cup_{k\ge 0}\M_k$. It is clear that $V$ is a $\p$-submodule of $M$ because identity (\ref{phibider}) holds. Now, consider $\g=\p\oplus V$ endowed with the natural bracket
$$[x+v,x'+v']_\g=[x,x']+\psi(x)v'-\psi(x')v,$$
for all $x,x'\in\p$, $v,v'\in V$. It is well known that $\g$  with such a bracket becomes a Lie algebra. Since $\p$ is perfect, it is clear that $\g$ will also be perfect if and only if $\psi(\p)V=V$. But, by construction,  $\M_k=\psi(\p)\M_{k-1}\subset \psi(\p)V$ for all $k\ge 1$ and the fact that $\p$ is perfect and (\ref{phibider}) imply that $\M_0\subset \psi(\p)\M_0,$ showing that $V\subset \psi(\p)V$.

Finally, let us define a symmetric product  on  $\g$ by
$(x+v)\circ (x'+v')=\varphi(x,x')$
for any $x,x'\in\p$, $v,v'\in V$. Notice that $\circ$ is nonzero since $\varphi$ is so. Using (\ref{phibider}) again, we then have for $x,x',y\in\p$ and $v,v',w\in V$ that
\begin{eqnarray*}
& & (y+w)\circ [x+v,x'+v']_\g=\varphi(y,[x,x'])=\psi(x)\varphi(y,x')-\psi(x')\varphi(y,x)\\
& & [(y+w)\circ(x+v),x'+v']_\g+[x+v,(y+w)\circ(x'+v')]_\g=\\
& & \phantom{==}[\varphi(y,x),x'+v']_\g+[x+v,\varphi(y,x')]_\g=-\psi(x')\varphi(y,x)+\psi(x)\varphi(y,x').
\end{eqnarray*}
This shows that $\circ$ is an $A_{BD}$-structure on the perfect Lie algebra $\g$.\end{proof}

\section{Some generalities on $A_{BD}$-structures}

Let us first show that we can reduce our study to the  case of an algebraically closed field. The proof of the following lemma is straightforward.

\begin{lemma}\label{complexif} Let $\K$ be a field of characteristic zero and let $\overline{\K}$ be its algebraic closure. Let $\g$ be a  Lie algebra over $\K$ and let $\circ$ be an $A_{BD}$-structure on $\g$. If $\overline{g}=\g\otimes_{\K}\overline{\K}$, then the product $\bullet$ on $\overline{g}$ defined by $(x_1\otimes \zeta_1)\bullet (x_2\otimes \zeta_2)=x_1\circ x_2\otimes \zeta_1\zeta_2$, for all $x_1,x_2\in\g$, $\zeta_1,\zeta_2\in\overline{\K}$  provides an $A_{BD}$-structure on $\overline{\g}$.
\end{lemma} 

\smallskip

 Let us consider a    Lie algebra with Levi decomposition $\mathfrak{g}=\mathfrak{s}\ltimes_\phi R(\mathfrak{g})$, where $R(\g)$ stands for the radical of $\g$, and   $\phi: \mathfrak{s} \to R(\mathfrak{g})$ is the representation of $\mathfrak{s}$ on $R(\mathfrak{g})$ defined by $\phi(a)= \ad(a)_{|R(\mathfrak{g}},$ for all $a \in \mathfrak{s}$. Suppose that $\g$ is endowed with an $A_{BD}$-structure $\circ$. This means that, for all $x$ element of $\g$, the multiplication $L_x: \g \to \g$ defined  by $L_x(y)= x\circ y,$ for all $y \in \g$, is a derivation of $\g$. Moreover, by Proposition 4.5 of \cite{BO},  $L_x(\g)$ is contained in $R(\mathfrak{g})$, for all $x\in \g.$ Thus, each map $L_x$ lies in the set $\mbox{Der}(\mathfrak{g}, R(\mathfrak{g}))$ of derivations of $\g$ with values in $R(\mathfrak{g})$. Using Lemma 5.7 of  \cite{BuDe}, the set $\mbox{Der}(\mathfrak{g}, R(\mathfrak{g}))$ can be characterized 
as follows:

\begin{lemma}\label{lemma21} Let $\g=\mathfrak{s}\ltimes_\phi R(\mathfrak{g})$ be the Levi decomposition of a Lie algebra $\g$ and
	let $D: \mathfrak{g}\to R(\mathfrak{g})$ be a linear map. Then $D\in Der(\mathfrak{g}, R(\mathfrak{g}))$ if and only if there exist two linear maps $d: R(\mathfrak{g}) \to R(\mathfrak{g})$ and $f: \mathfrak{s} \to R(\mathfrak{g})$ such that
	\begin{enumerate}
		\item[{\rm (1)}] $D(a+r)=f(a)+d(r),$
		\item[{\rm (2)}]  $d([r,r'])=[d(r),r']+[r,d(r')]$, 
		\item[{\rm (3)}] $f([a,a'])=\phi(a)f(a')-\phi(a')f(a)$, 
		\item[{\rm (4)}] $d\phi(a)(r)=\phi(a)d(r)+[f(a),r]$,
	\end{enumerate}
	for all $a,a' \in \mathfrak{s}, \, r,r' \in R(\mathfrak{g})$.
\end{lemma}

 \smallskip

\begin{remark} \label{rema1} {\em Notice that condition (2) means that $d$ is a derivation of $R(\g)$ and condition (3) is equivalent to $f\in Z^1(\mathfrak{s},R(\g))$. Thus, we can apply the first Whitehead's lemma \cite{jacob} to the cocycle $f$ to assure the existence of a vector $r_f\in R(\g)$ such that $f(a)=\phi(a)r_f$ for all $a\in {\mathfrak s}$.}
\end{remark}

\smallskip

All through the paper the notation $\g=\mathfrak{s}\ltimes_\phi R(\mathfrak{g})$ will stand for a Levi decomposition of the Lie algebra $\g$ over an algebraically closed field $\K$. Notice that if $\g$ admits a nontrivial $A_{BD}$-structure, then $R(\g)\ne\{0\}$ since semisimple Lie algebras do not admit such structures \cite[Prop. 4.3]{BO}.

\smallskip

\begin{proposition}\label{proposition21}
	Let $\mathfrak{g}=\mathfrak{s}\ltimes_\phi R(\mathfrak{g})$ be a  Lie algebra. A product   $\circ$ on $\g$ defines an $A_{BD}$-structure if and only if there exist  linear maps $F:\mathfrak{s}\to R(\mathfrak{g})$, $G:R(\mathfrak{g})\to R(\mathfrak{g})$ and $\Delta:R(\mathfrak{g})\to \mbox{\rm Der}(R(\mathfrak{g}))$ such that $$(a+r)\circ (a'+r')=\phi(a)Fa'+\phi(a)Gr'+\phi(a')Gr+\Delta(r)(r'),$$
	where  $F, G$ and $\Delta$ verify
	\begin{enumerate}
		\item[{\rm (1)}] $\phi(a)Fa'=\phi(a')Fa,$
		\item[{\rm (2)}] $\phi(a)G\in\mbox{\rm Der}(R(\mathfrak{g}))$,
									\item[{\rm (3)}] $\Delta(r)(r')=\Delta(r')(r),$
									\item[{\rm (4)}] $\phi(a)G\phi(a')(r')=\phi(a')\phi(a)G(r')+[\phi(a)Fa',r'],$
								\item[{\rm (5)}] $\Delta(r)\phi(a')(r')=\phi(a')\Delta (r)(r')+[\phi(a')G(r),r'],$
	\end{enumerate}
	for all $a,a'\in \mathfrak{s}$, $r,r' \in R(\mathfrak{g})$.
\end{proposition}
\begin{proof}
	Suppose that $\circ$ is an $A_{BD}$-structure. Let us consider $x \in \mathfrak{g}$. Since $L_x \in \mbox{\rm Der}(\mathfrak{g},R(\mathfrak{g}))$, by Lemma \ref{lemma21}, there exist a linear map $f_x:\mathfrak{s}\to R(\mathfrak{g})$ and a derivation $d_x:R(\mathfrak{g}) \to R(\mathfrak{g})$ such that $L_x(a+r)=f_x(a)+d_x(r)$, verifying $f_x([a,a'])=\phi(a)f_x(a')-\phi(a')f_x(a)$ and $d_x\phi(a)(r)=\phi(a)d_x(r)+[f_x(a),r]$, for all $a,a'\in \mathfrak{s}$, $r\in R(\mathfrak{g})$.
	Using Whitehead's Lemma as in Remark \ref{rema1}, we see that there exists $v_{f_x}\in R(\mathfrak{g})$ such that $f_x(a)=\phi(a)v_{f_x}$ for all $a\in {\mathfrak s}$. Let us then define $F:\s\to R(\g)$ by $F(a)=v_{f_a}$ for $a\in\s$ and $G:R(\g)\to R(\g)$ by $G(r)=v_{f_r}$ whenever $r\in R(\g)$.

One also has  $d_a(r)=L_a(r)=L_r(a)=f_r(a)=\phi(a)G(r)$, for all $a\in {\mathfrak s}$, $r\in R(\g)$, and, as a result, $L_a(a'+r')=\phi(a)F(a')+\phi(a)G(r')$ for all $a,a'\in \mathfrak{s}$, $r'\in R(\mathfrak{g})$. Clearly, the fact that $d_a$ is a derivation of $R(\g)$ proves (2). Besides, if we define define $\Delta(r)=d_r\in \mbox{\rm Der}(R(\mathfrak{g}))$ for all $r\in R(\g)$, we have $L_r(a'+r')=L_r(a')+\Delta(r)(r')=\phi(a')Gr+\Delta(r)(r')$ and,  hence
$$(a+r)\circ (a'+r')=\phi(a)Fa'+\phi(a)Gr'+\phi(a')Gr+\Delta(r)(r'),$$	
for all $a,a'\in \mathfrak{s}$, $r,r'\in R(\mathfrak{g})$, as claimed.
Notice that the commutativity of $\circ$ clearly implies assertions (1) and (3). 

	Finally recall that we had  $d_x\phi(a')(r')=\phi(a')d_x(r')+[f_x(a'),r']$, for all $x\in\g$,  $a'\in \mathfrak{s}$, $r'\in R(\mathfrak{g})$. When $x=a\in\s$ we  then get
	$\phi(a')G\phi(a)(r')=\phi(a')\phi(a)G(r')+[\phi(a)Fa',r'],$ proving identity (4). Taking $x=r\in R(\g)$ we obtain $\Delta(r)\phi(a')(r')=\phi(a')\Delta(r)(r')+[\phi(a')G(r),r']$, which proves (5).
	
	For the converse, if $F, G$ and $\Delta$ verifying equalities (1) through (5) are given and we define $(a+r)\circ (a'+r')=\phi(a)Fa'+\phi(a)Gr'+\phi(a')Gr+\Delta(r)(r'),$ for all $a,a'\in \mathfrak{s}, \, r \in R(\mathfrak{g})$, it is a simple calculation to see that $\circ$ is commutative and that the maps $f_{a+r}:\s\to R(\g)$ and $d_{a+r}: R(\g)\to R(\g)$ defined by $f_{a+r}(a')=\phi(a)Fa'+\phi(a')Gr$, $d_{a+r}(r')=\phi(a)Gr'+\Delta(r)(r')$, for all $a,a'\in \mathfrak{s}$, $r,r' \in R(\mathfrak{g})$ verify the conditions (2) through (4) of Lemma \ref{lemma21}. This shows that $L_{a+r}$ is a derivation of the Lie algebra $\g$, for any $a\in \mathfrak{s}$, $r\in R(\mathfrak{g})$, which proves that $\circ$ defines an $A_{BD}$-structure.
\end{proof}

\smallskip

As we mentioned in the introduction, the case with abelian radical will be important in our study. In that case we immediately deduce form Proposition \ref{proposition21} the following characterization of $A_{BD}$-structures:

\smallskip

\begin{corollary}\label{corol21}
	Let $\mathfrak{g}=\mathfrak{s}\ltimes_\phi R(\mathfrak{g})$ be  a  Lie algebra with $ R(\mathfrak{g})$ abelian. A product   $\circ$ on $\g$ defines an $A_{BD}$-structure if and only if there exist  linear maps $F:\mathfrak{s}\to R(\mathfrak{g})$, $G:R(\mathfrak{g})\to R(\mathfrak{g})$ and $\Delta:R(\mathfrak{g})\to {\mathfrak{gl}}(R(\mathfrak{g}))$ such that $$(a+r)\circ (a'+r')=\phi(a)Fa'+\phi(a)Gr'+\phi(a')Gr+\Delta(r)(r'),$$
	where  $F, G$ and $\Delta$ verify
	\begin{enumerate}
		\item[{\rm (1)}] $\phi(a)Fa'=\phi(a')Fa,$
		\item[{\rm (2)}] $\Delta(r)(r')=\Delta(r')(r),$
		\item[{\rm (3)}] $\phi(a)G\phi(a')(r)=\phi(a')\phi(a)G(r),$
		\item[{\rm (4)}] $\Delta(r)\phi(a)(r')=\phi(a)\Delta (r)(r'),$
	\end{enumerate}
	for all $a,a'\in \mathfrak{s}$, $r,r' \in R(\mathfrak{g})$.
\end{corollary}

\smallskip

\begin{remark}\label{rema22} {\em Let us notice that there are three interesting particular cases of the situation described in Corollary \ref{corol21}. For a Lie algebra $\mathfrak{g}=\mathfrak{s}\ltimes_\phi R(\mathfrak{g})$ with abelian radical with the notations above, we have:
\begin{enumerate}
\item[(1)] If we take $F\ne 0$ such that $\phi(a)Fa'=\phi(a')Fa,$ for all $a,a'\in\s$ and there exists $a_0\in\s$ with $\phi(a_0)F\ne 0$, this is to say $F(\s)\not\subset R(\g)^\s$, then the map
$(a+r)\circ (a'+r')=\phi(a)Fa'$, for all $a,a'\in \mathfrak{s}$, $r,r' \in R(\mathfrak{g})$, is a nontrivial $A_{BD}$-structure on $\g$.
\item[(2)] If we take $G\ne 0$ such that $\phi(a)G\phi(a')(r)=\phi(a')\phi(a)G(r),$ for all $a,a'\in\s$, $r\in R(\g)$ and $G(R(\g))\not\subset R(\g)^\s$, then the map
$(a+r)\circ (a'+r')=\phi(a)Gr'+\phi(a')Gr$, for all $a,a'\in \mathfrak{s}$, $r,r' \in R(\mathfrak{g})$, is a nontrivial $A_{BD}$-structure on $\g$.
 \item[(3)] If we take $\Delta\ne 0$ such that $\Delta(r)(r')=\Delta(r')(r)$ and $\Delta(r)\phi(a)(r')=\phi(a)\Delta (r)(r'),$ for all $a\in\s$, $r,r'\in R(\g)$  then the map
$(a+r)\circ (a'+r')=\Delta(r)(r')$, for all $a,a'\in \mathfrak{s}$, $r,r' \in R(\mathfrak{g})$, is a nontrivial $A_{BD}$-structure on $\g$.
\end{enumerate}
}
\end{remark}

\section{Reduction to the case of abelian radical}\label{reduc}

We will devote this section to prove that if an arbitrary perfect Lie algebra admits a nonzero $A_{BD}$-structure, then one can find a perfect Lie algebra with abelian radical also admitting a nontrivial $A_{BD}$-structure.

\begin{proposition}\label{proposition22}
	Suppose that a perfect Lie algebra $\mathfrak{g}=\mathfrak{s}\ltimes_\phi R(\mathfrak{g})$ with non-abelian radical admits a nontrivial $A_{BD}$-structure. If $R(\mathfrak{g})$ does not have $1$-dimensional $\mathfrak{s}$-submodules, then the Lie algebra with abelian radical $\tilde{\mathfrak{g}}=\mathfrak{s} \ltimes_\phi \mathfrak{a},$ where $\mathfrak{a}= R(\mathfrak{g})$ as vector spaces, is perfect and also admits a nontrivial $A_{BD}$-structure.
\end{proposition}
\begin{proof}
	The $A_{BD}$-structure on $\mathfrak{g} = \mathfrak{s} \ltimes_{\phi} R(\mathfrak{g})$, We know that there exist linear maps $F : \mathfrak{s} \to R(\g)$, $G : R(\g) \to R(\g)$ and $\Delta: R(\g)\to \mbox{Der}(R(\g))$  verifying (1)-(5) in Proposition \ref{proposition21} such that the $A_{BD}$ product  is given by:
$$
		(a + r) \circ (a' + r') = \phi(a)F(a') + \phi(a)Gr'  + \phi(a')Gr +  \Delta(r)(r'), \quad \forall a, a' \in \mathfrak{s}, \ r, r' \in R(\mathfrak{g}),
$$
		
	Notice that $\tilde{\mathfrak{g}}$ is also perfect because  $R(\mathfrak{g})$ does not admit 1-dimensional $\mathfrak{s}$-submodules. If the above map $F$  is nonzero, then   by (1) in Remark \ref{rema22} we obtain a nontrivial $A_{BD}$-structure on $\tilde{\mathfrak{g}}.$ Notice that the condition $F(\s)\not\subset R(\g)^\s$ trivially holds because   $R(\mathfrak{g})$ does not have $1$-dimensional $\mathfrak{s}$-submodules.
	
 Suppose now that $F = 0$ but $G\ne 0$. One has again $G(R(\g))\not\subset R(\g)^\s=\{0\}$ and condition (4) in Proposition \ref{proposition21} now reads
$\phi(a)G\phi(a')(r)=\phi(a')\phi(a)G(r)$, showing that $(a+r)\circ (a'+r')=\phi(a)Gr'+\phi(a')Gr$, for all $a,a'\in \mathfrak{s}$, $r,r' \in R(\mathfrak{g})$, provides a nontrivial $A_{BD}$-structure on $\g$, as in the second case of Remark \ref{rema22}.

Finally, if $F=0$ and $G=0$, we must have $\Delta\ne 0$. It is clear that condition $\Delta(r)(r')=\Delta(r')(r)$ holds for all $r,r'\in R(\g)$ and, since $G=0$, the condition (5) of Proposition \ref{proposition21} gives $\phi(a)\Delta(r) = \Delta(r)\phi(a)$, for all $a,a'\in \mathfrak{s}$, $r\in R(\g)$. Thus the product $\circ$ is also an $A_{BD}$-structure on $\tilde\g$. 	
\end{proof}

\medskip

\begin{proposition}\label{proposition23}
	Let $\mathfrak{g}=\mathfrak{s}\ltimes_\phi R(\mathfrak{g})$ be a perfect Lie algebra with  a non-trivial $A_{BD}$-structure $\circ$. 
	Denote $\mathfrak{n}_1=[\mathfrak{s}, R(\mathfrak{g})]$ and $\mathfrak{n}_2=R(\mathfrak{g})^\mathfrak{s}$ and  put $\phi_1=\phi_{\vert_{\mathfrak{n}_1}}$. 
	
	If either $\s\circ\s\ne\{0\}$ or $\s\circ\mathfrak{n}_1\ne\{0\}$, then the Lie algebra $\tilde{\mathfrak{g}}=\mathfrak{s}\ltimes_{\phi_1} R(\tilde{\mathfrak{g}})$ with abelian radical  $R(\tilde{\mathfrak{g}})$, where $R(\tilde{\mathfrak{g}})=\mathfrak{n}_1$ as vector spaces, is perfect and also admits a non-trivial $A_{BD}$-structure.
\end{proposition}
\begin{proof}
It is quite clear that $\tilde{\mathfrak{g}}=\mathfrak{s}\ltimes_{\phi_1} \mathfrak{n}_1$ is perfect because $\mathfrak{n}_1=[\mathfrak{s}, R(\mathfrak{g})]=[\mathfrak{s},\mathfrak{n}_1].$ Let us denote $\pi_1: R(\g)=\mathfrak{n}_1\oplus \mathfrak{n}_2\to \mathfrak{n}_1$  the natural projection onto $\mathfrak{n}_1$. 

Take $F$, $G$ and $\Delta$ as in Proposition \ref{proposition21}, so that  
$$
	(a + r) \circ (a' + r') = \phi(a)F(a') + \phi(a')G(r) + \phi(a)G(r') + \Delta(r)(r')$$ for all $a, a' \in \mathfrak{s}$, $r, r' \in R(\mathfrak{g})$.
	Notice that $\s\circ\s\subset [\s,R(\g)]=\mathfrak{n}_1$ and $\s\circ\mathfrak{n}_1\subset  [\s,R(\g)]=\mathfrak{n}_1$.  
If $\s\circ\s\ne\{0\}$, then $F(\s)\not\subset \mathfrak{n}_2$ and we can construct a nontrivial $A_{BD}$-structure in $\tilde\g$ as in the first case of Remark \ref{rema22} for $F_1=\pi_1F$. Otherwise, when $\s\circ\s=\{0\}$, we can consider $G_1=\pi_1G_{|\mathfrak{n}_1}$ and define an $A_{BD}$-structure in $\tilde\g$ by
$(a+r)\bullet(a'+r')=\phi_1(a')G_1(r)+\phi_1(a)G_1(r'), $ for any $a,a'\in \mathfrak{s}$, $r,r'\in \mathfrak{n}_1$, as we have done in the second example of Remark \ref{rema22}. It is obviously nontrivial because $\mathfrak{s}\circ \mathfrak{n}_1 \neq \lbrace 0 \rbrace$ and
$a\bullet r_1=\phi_1(a)G_1r_1=\phi(a)\pi_1Gr_1=a\circ r_1,$
for all $a\in\s$, $r_1\in \mathfrak{n}_1$. This completes the proof.
\end{proof}
\begin{lemma}\label{lemma23}
	Suppose that  $\mathfrak{g} = \mathfrak{s} \ltimes_{\phi} R(\mathfrak{g})$ is a perfect Lie algebra with non-abelian radical that admits a non-trivial $A_{BD}$-structure  $\circ$.  
	Let $R(\mathfrak{g})= \mathfrak{n}_1 \oplus \mathfrak{n}_2$ be the decomposition into $s$-submodules with $\mathfrak{n}_1=[\mathfrak{s}, R(\mathfrak{g})]$ and  $\mathfrak{n}_2 = R( \g)^\s$ and put $\phi_1 = \phi_{|_{\mathfrak{n}_1}}$.  If $\mathfrak{s} \circ \mathfrak{n}_1 = \{0\},$  then $\g\circ \mathfrak{n}_1 = \{0\}$.  
\end{lemma}
\begin{proof} Put that the product on  $\g=\s\ltimes_\phi R(\g)$ is given as in Proposition \ref{proposition21} by
$$(a + r) \circ (a' + r') = \phi(a)F(a') + \phi(a')G(r) + \phi(a)G(r') + \Delta(r)(r'),\quad a,a'\in\s,\ r,r'\in R(\g).$$
Since $\g$ is perfect, one must have $\mathfrak{n}_2\subset[\mathfrak{n}_1 ,\mathfrak{n}_1 ]$ and we can, thus, find a set of generators of $R(\g)$ (as Lie algebra)  $\{t_\alpha\}_{\alpha\le \ell}\subset\mathfrak{n}_1 $. But then, since $\phi(a)G(t_\alpha)=0$ and $\phi(a)G$ is a derivation of $R(\g)$ for all $a\in\s$, we also get $\phi(a)G(r_2)=0$ for all $r_2\in\mathfrak{n}_2$. Hence, $\s\circ \mathfrak{n}_2=\{0\}$.

As a consequence, we get $r_2\circ (s+r)=\Delta(r_2)(r)$ for all $s\in\s$, $r\in R(\g)$, $r_2\in\mathfrak{n}_2$. Notice that the equality $\phi(a)G=0$, for all $a\in\s$, and the condition (5) of Proposition \ref{proposition21} give $\phi(a)\Delta(r)(r_2)=\Delta(r)\phi(s)r_2=0,$
which shows that $\Delta(r)(r_2)\in\mathfrak{n}_2$ for all $r\in R(\g)$, $r_2\in\mathfrak{n}_2$. But then
$$0=\phi(a)\Delta({r_2})(r)=\Delta({r_2})\phi(s)(r),$$
for all $r\in R(\g)$, which implies $\Delta({r_2})(r_1)=0$ for all $r_1\in\mathfrak{n}_1=\phi(\s)(\mathfrak{n}_1)$. But, again, since $\Delta({r_2})$ is a derivation and vanishes on the set of generators $\{t_\alpha\}_{\alpha\le \ell}\subset\mathfrak{n}_1 $, one has that $\Delta({r_2})=0$.
\end{proof}
\begin{proposition}\label{propreduc}
	If there exists a perfect Lie algebra admitting a non-trivial $A_{BD}$-structure then there exists a perfect Lie algebra with abelian radical also admitting a non-trivial $A_{BD}$-structure.
\end{proposition}
\begin{proof}
	Let $\mathfrak{g}=\mathfrak{s}\oplus R(\mathfrak{g})$ be a perfect Lie algebra with a non-trivial $A_{BD}$-structure  $\circ$. We have the decomposition of the $\mathfrak{s}$-module $R(\g)$ into the sum of two submodules given by
 $R(\mathfrak{g})=\mathfrak{n}_1\oplus \mathfrak{n}_2,$ where  $\mathfrak{n}_1=[\mathfrak{s}, R(\mathfrak{g})]$ and $\mathfrak{n}_2=R(\mathfrak{g})^\mathfrak{s}.$ The case $\mathfrak{n}_2=\lbrace0 \rbrace$ was considered  in Proposition \ref{proposition22}, so  we can suppose that $\mathfrak{n}_2\neq \lbrace 0\rbrace$.

According to Proposition \ref{proposition23}, if either $\s\circ\s\ne\{0\}$ or $\s\circ \mathfrak{n}_1\ne\{0\}$, then we can construct an $A_{BD}$-structure on the perfect Lie algebra $\tilde\g=\s\ltimes_{\phi_1} R(\tilde\g)$ where $R(\tilde\g)$ is the vector space $\mathfrak{n}_1$ with zero bracket and $\phi_1$ is the representation on $\mathfrak{n}_1$ defined by $\phi$. Therefore, we can suppose that $\s\circ(\s\oplus\mathfrak{n}_1)=\{0\}$ so that, using Lemma \ref{lemma23}, we get that our product $\circ$ is simply given by
$$(a+r_1+r_2)\circ(a'+r_1'+r_2')=\Delta({r_1})(r_1'), \quad  a,a'\in \mathfrak{s}, r_1,r_1'\in \mathfrak{n}_1,  r_2,r_2'\in \mathfrak{n}_2.$$
Let $\pi_1:R(\g)=\mathfrak{n}_1\oplus \mathfrak{n}_2\to \mathfrak{n}_1$ denote the natural projection and define on the perfect Lie algebra with abelian radical $\tilde\g$ given above the product
$$(a+r_1)\bullet(a'+r_1')=\pi_1\Delta({r_1})(r_1'), \quad  a,a'\in \mathfrak{s}, r_1,r_1'\in \mathfrak{n}_1.$$
It is straightforward to prove that the restriction $\Delta_1$ of $\pi_1\Delta$ to $\mathfrak{n}_1\times \mathfrak{n}_1$ verifies the conditions $\Delta_1({r_1})(r_1')=\Delta_1({r_1'})(r_1)$ and $\Delta_1(r_1)\phi(a)(r_1')=\phi(a)\Delta_1 (r_1)(r'_1),$ for all $a\in\s$, $r_1,r_1'\in \mathfrak{n}_1$, because they hold for $\Delta$. So, according to the third case of Remark \ref{rema22}, it suffices to show that $\Delta_1\ne 0$. But $\Delta_1= 0$ would imply $\Delta(r_1)(r_1')\in \mathfrak{n}_2$ for all $r_1,r_1'\in \mathfrak{n}_1$ and the condition (5) in Proposition \ref{proposition21} would give
$$0=\phi(a)\Delta(r_1)(r_1')=\Delta(r_1)\phi(a)(r_1'),$$
which, since $\mathfrak{n}_1=\phi(\s)\mathfrak{n}_1$, would imply $\Delta=0$, a contradiction since $\circ$ is nontrivial.
\end{proof}

\section{The case  with abelian radical and the main result}\label{abrad}

It was proved in \cite[Th. 4.1]{BO} that if the radical of $\g$ is abelian and a nonzero simple ${\mathfrak s}$-module then every $A_{BD}$-structure on $\g$ is defined by $(s+r)\circ (s'+r')=[\beta(s),s']$, for all  $a,a'\in \mathfrak{s}$, $r,r'\in R(\mathfrak{g})$, where $\beta:{\mathfrak s}\to \g$ is a map such that $[\beta(s),s']=[\beta(s'),s]$. Using also \cite[Prop. 4.3]{BO} one easily shows that $\beta(\mathfrak{s})\subset R(\g)$ so that we see that, in that case, the only possible $A_{BD}$-structures are those given as in the first case of Remark \ref{rema22}. Explicitly, we can reformulate the mentioned theorem in \cite{BO} as follows:

\begin{proposition}
	\label{SaidHassan}
	Let $\mathfrak{g} =  \mathfrak{s} \ltimes_\phi R(\mathfrak{g})$ be a  Lie algebra such that $R(\mathfrak{g}) \neq \{0\}$ is an abelian Lie subalgebra and $R(\mathfrak{g})$ is a simple $\mathfrak{s}$-module. Every $A_{BD}$-structure on $\g$ is defined by $(a+r)\circ (a'+r')=\phi(a)F(a')$, for all $a,a'\in \mathfrak{s}$, $r,r'\in  R(\mathfrak{g})$, where $F: \mathfrak{s} \to  R(\mathfrak{g})$ is a linear map such that $\phi(a)F(a')=\phi(a')F(a)$, for all $a,a'\in \mathfrak{s}$.
	\end{proposition}

Let us prove that all $A_{BD}$-structures are trivial in the particular case $\mathfrak{s} =\mathfrak{sl}_2(\mathbb{C})$ and $R(\mathfrak{g})=V(m)$ for a certain $m\in {\mathbb N}\backslash \{0\}$, the $(m+1)$-dimensional irreducible module endowed with a vanishing bracket. Such result was already seen in  \cite[Corollary 4.2]{BO} for an even $m$. 
 
 \begin{corollary}\label{msimple} For $m\in {\mathbb N}\backslash\{0\}$, let $\phi$ denote the simple representation of $\mathfrak{sl}(2)$ on the $(m+1)$-dimensional space $V(m)$ and  consider that  $\g=\mathfrak{sl}(2)\ltimes_\phi V(m)$, where $R(\mathfrak{g})= V(m)$ is abelian. Any $A_{BD}$-structure on $\g$ is trivial.
 \end{corollary}

\begin{proof}
Let us consider the basis $\{h,e,f\}$ of  $\mathfrak{sl}(2)$ satisfying 
$[h,e]= 2e$, $[h,f]= -2f$, $[e,f]= h.$ We will use the description of the irreducible representations of $\mathfrak{sl}(2)$ as given in \cite{Serre}, so that there exists a basis $\{v_i,\, 0 \leq i \leq m\}$ of $V(m)$ such that 
\begin{eqnarray*}
	& & \phi(h)v_i=(m-2i)v_i,\quad  i\in \{0,\dots,m\},\\
	& &\phi(e)v_i=(m-i+1)v_{i-1},\quad i\in \{1,\dots,m\},\\
	& &\phi(f)v_i=(i+1)v_{i+1},\quad i\in \{0,\dots,m-1\},\\
	& &\phi(e)v_0=0,\quad \phi(f)v_m=0.
\end{eqnarray*}

Let $F: \mathfrak{sl}(2) \to  V(m)$ be a linear map such that $\phi(a)F(a')=\phi(a')F(a)$, for all $a,a'\in \mathfrak{sl}(2)$ and put
$$F(e)=\sum_{i=0}^ma_iv_i,\quad F(f)=\sum_{i=0}^mb_iv_i, \quad 
	F(h)=\sum_{i=0}^mc_iv_i.$$
The conditions $\phi(e)F(f)=\phi(f)F(e)$, $\phi(e)F(h)=\phi(h)F(e)$ and $\phi(f)F(h)=\phi(h)F(f)$  respectively give
\begin{eqnarray*}
	& & b_1=a_{m-1}=0,  \quad ia_{i-1}=(m-i)b_{i+1} \quad 1\leq i\leq m-1,\\
	& &a_m=0,  \quad (m-i)c_{i+1}=(m-2i)a_i \quad 0\leq i\leq m-1,\\
	& &b_0=0, \quad ic_{i-1}=(m-2i)b_i \quad 1\leq i\leq m.
\end{eqnarray*}
and from those equations we deduce that $c_0=c_m=0$ and
$$\left((m-2i+2)(m-i)(i+1)-(m-2i-2)(m-i+1)i\right)c_i=0, \quad 1\leq i\leq m-1.$$
Since we have that $$(m-2i+2)(m-i)(i+1)-(m-2i-2)(m-i+1)i]=m^2+2m\ne 0,$$ one finally gets $c_i=0$ for all $i=0,\dots,m$.  Now, the identities 
$$0=(m-i)c_{i+1}=(m-2i)a_i \quad (0\leq i\leq m-1),\qquad 0=ic_{i-1}=(m-2i)b_i\quad  (1\leq i\leq m),$$
show that $a_i=b_i=0$ whenever $i\ne m/2$. If $m$ is even and $m\ge 4$, for $k=m/2$ we have $1\le k-1 < k+1\le m-1$ and, from $ia_{i-1}=(m-i)b_{i+1}$, whenever $1\leq i\leq m-1$, we then get
$$0=(k-1)a_{k-2}=(m-k+1)b_k,\quad (k+1)a_k=(m-k-1)b_{k+2}=0,$$
which imply $b_k=a_k=0$ because $(m-k+1)=k+1\ne 0$. If $m=2$ then $m/2=1$ and we already had
$b_1=0$, $a_1=a_{m-1}=0$. 
Thus, we finally get that $F$ is identically zero. According to Proposition \ref{SaidHassan} we then have that any $A_{BD}$-structure on $\g$ is trivial.
\end{proof}	

\medskip

We can now extend our result to perfect Lie algebras with arbitrary abelian radical and $\mathfrak{sl}(2)$ as Levi component. Explicitly:

\smallskip

\begin{proposition}\label{sl2}
 Let $\mathfrak{g}= \mathfrak{s}\ltimes_\phi R(\mathfrak{g})$ be a perfect Lie algebra with an abelian radical $R(\mathfrak{g})$ and $\mathfrak{s}=\mathfrak{sl}(2)$.  Then, any $A_{BD}$-structure on $\mathfrak{g}$ is trivial.
 \end{proposition}

\begin{proof}
Let $R(\mathfrak{g}) = M_1 \oplus \dots \oplus M_m$  be the decomposition of the $\mathfrak{sl}(2)$-module $R(\mathfrak{g})$ into a direct sum of simple $\mathfrak{sl}(2)$-submodules $M_i$, $i\in\{1,\dots,m\}$. Notice that the fact that $\mathfrak{g}$ is perfect implies that the dimension of $M_i$ is is greater than or equal to $2$. Let us denote by $\pi_i:R(\g)\to M_i$ the corresponding projection on each $i\le m$.

Let us  consider  an $A_{BD}$-structure $\circ$ on $\mathfrak{g}$ and
let $F:\mathfrak{sl}(2)\to R(\g)$, $G: R(\g)\to R(\g)$ and $\Delta: R(\g)\to {\mathfrak{gl}}(R(\g))$ be  as given in Corollary \ref{corol21}, so that
$$(a+r)\circ (a'+r')=\phi(a)F(a')+\phi(a)G(r')+\phi(a')G(r)+\Delta(r)(r'),$$
for all $a,a'\in \mathfrak{sl}(2)$ and $r,r'\in R(\g)$.

If $F$ is not identically zero, there exist $a\in \mathfrak{sl}(2)$ and $j\le m$ such that $F_j(a)=\pi_{j}F(a)\ne 0$. But, if $\phi_j(a)$ denotes the restriction of $\phi(a)$ to $M_j$, we then have
$$\phi_j(a')F_j(a)=\pi_j\phi(a')F(a)=\pi_j\phi(a)F(a')=\phi_j(a)\pi_jF(a')=\phi_j(a)F_j(a'),$$ for all $a,a'\in \mathfrak{sl}(2)$, but then,  as shown in the first case of Remark \ref{rema22}, the Lie algebra $\g_j=\mathfrak{sl}(2)\ltimes_{\phi_j} M_j$ would admit a nontrivial $A_{BD}$-structure, a contradiction with Corollary \ref{msimple}. Thus, $F=0$.   

Let us now see that $G$ is also identically null. Recall that, as $R(\g)$ is abelian, from assertion (3) in Corollary \ref{corol21} we get $\phi(b)\phi(a)G=\phi(a)G\phi(b)$ for all $a,b\in \mathfrak{sl}(2)$. Thus, if for $i,j \in \lbrace 1,\dots,m\rbrace$ we denote 
$D^a_{ij}=\pi_j \phi(a) G_{\vert_{M_i}}= \phi(a) \pi_j G_{\vert_{M_i}},$ one easily sees that $[b,D^a_{ij}(r_i)]=D^a_{ij}([b,r_i])$ for all $a,b\in \mathfrak{sl}(2)$, $r_i\in M_i$. This implies that $D^a_{ij}$ is a homomorphism of simple $\mathfrak{sl}(2)$-modules and, thus, either it is zero or it is bijective. But, as $D^a_{ij}(r_i)=\phi(a) \pi_j G(r_i)\in \phi(a)(M_j)$, one gets that if an element $a\in \mathfrak{sl}(2)$ verifies
 $[a,M_j] \ne M_j$ for all $j\le m$, then $D^a_{ij}=0$ for all $i,j$ and hence $\phi(a)G=0$. Let us prove that this implies that $G$ must be zero. Let
 $\lbrace h,e,f\rbrace$ be the standard   basis of $\mathfrak{sl}(2)$ with $[h,e]= 2e$, $[h,f]= -2f$, $[e,f]= h.$  We must have $\phi(e)G=0$ and $\phi(f)G=0$ because the restrictions of $\phi(e)$ and $\phi(f)$ to each $M_j$ are nilpotent maps so that $\phi(e)M_j\ne M_j$ and $\phi(f)M_j\ne M_j$. But then we also have $\phi(h)G=0$ because $\phi(h)=\phi(e)\phi(f)-\phi(f)\phi(e)$. As a consequence, $\phi(a)G=0$ for all $a\in \mathfrak{sl}(2)$, which can only occur when $G=0$. 

Notice that, since $F=0$ and $G=0$,  we then get $(a+r)\circ (a'+r')=\Delta(r)(r'),$
for all $a,a'\in \mathfrak{sl}(2)$ and $r,r'\in R(\g)$.
Let us consider $N_0=\{ r\in R(\g)\ :\ \Delta(r)=0\}$. It is clear that $N_0$ is a $\mathfrak{sl}(2)$-submodule of $R(\g)$ because $\mathfrak{sl}(2)\circ\g=\{0\}$ and left multiplications for $\circ$  are  derivations of the Lie algebra.   
Let us suppose that $N_0\ne R(\g)$. We can then find nonzero simple $\mathfrak{sl}(2)$-submodules $N_1,\dots, N_p$ of $R(\g)$ such that $R(\g)=N_0\oplus  N_1\oplus\cdots\oplus N_p$. Recall that $\dim(N_j)\ge 2$ and that, if $\lbrace h,e,f\rbrace$ is again  the standard   basis of $\mathfrak{sl}(2)$, then $\dim(\mbox{Ker}(\phi_j(e)))=1$ for all  $j\in\{1,\dots, p\}$, where $\phi_j$ denotes the restriction of $\phi(e)$ to each $N_j$. Let us fix $\ell\in\{1,\dots, p\}$ and choose a nonzero vector $r_\ell\in \mbox{Ker}(\phi_\ell(e))$. For any $i,j\in\{1,\dots, p\}$, define $\Gamma_{ij}:N_i\to N_j$ by $\Gamma_{ij}=\pi_j\Delta(r_\ell)_{|N_i}$. The assertion (4) of
Corollary \ref{corol21} gives $\phi(a)\Gamma_{ij}=\Gamma_{ij}\phi(a)$ for all $a\in \mathfrak{sl}(2)$, showing that $\Gamma_{ij}$ is a morphism of simple $\mathfrak{sl}(2)$-modules and, again, either it must be zero or invertible. But for all $r\in M_i$,  using once more the equalities (2) and (4) of
Corollary \ref{corol21}, we have
$$\phi_j(e)\Gamma_{ij}(r)=\pi_j\phi(e)\Delta(r_\ell)(r)=\pi_j\phi(e)\Delta(r)(r_\ell)=\pi_j\Delta(r)\phi(e)(r_\ell)=0.$$
This implies that $\Gamma_{ij}(M_i)\subset \mbox{Ker}(\phi_j(e))$, so that $\dim (\Gamma_{ij}(M_i))\le 1$, which together with the fact that $\dim(N_j)\ge 2$ shows that  $\Gamma_{ij}$ cannot be invertible and must vanish for all $1\le i,j\le m$. But this would imply $\Delta(r_\ell)=0$ and, therefore, $r_\ell\in  N_0\cap N_\ell$, a contradiction since we chose $r_\ell\ne 0$.

 So we  conclude that our assumption $N_0\ne R(\g)$ is false and consequently $\g \circ \g= \{0\}.$
 \end{proof}

\begin{remark} {\em The non-existence of $A_{BD}$-structures on finite-dimensional perfect Lie algebras of the form $\mathfrak{g}= \mathfrak{sl}(2)\ltimes_\phi R(\mathfrak{g})$ with abelian radical is in great contrast with the infinite-dimensional case;  notice that the infinite-dimensional example given in \cite{BO} has  $ \mathfrak{sl}(2)$ as its Levi summand.

Proposition \ref{sl2} is essential in the proof of our next result which generalizes Proposition \ref{SaidHassan} to an arbitrary perfect Lie algebra with abelian radical.}
\end{remark}

\begin{proposition} \label{le2} Let $\g=\s\ltimes_\phi R(\g)$ be a perfect Lie algebra with abelian radical. If $\g$ admits an $A_{BD}$-structure $\circ$, then $R(\g)\circ \g=\{0\}$.

As a consequence, $\mathfrak{g}\circ \mathfrak{g}\circ \mathfrak{g}=\lbrace 0 \rbrace$ and there exists a linear map $F: \mathfrak{s}\to R(\mathfrak{g})$ such that $\phi(a)F(a')=\phi(a')F(a)$ for all $a,a \in \mathfrak{s}$ and 
	$(a+r)\circ (a'+r')=\phi(a)F(a')$ for any $ a,a'\in \mathfrak{s}$, $r,r'\in R(\mathfrak{g}).$
\end{proposition}	

\begin{proof} Let $F, G$ and $\Delta$ be the maps defined for $\circ$ as in Corollary \ref{corol21}. 
	Suppose that there exists $r_0\in R(\mathfrak{g})$ such that ${r_0}\circ \g\neq \{0\}$. Suppose first that $\Delta(r_0)\ne 0$ and take $r_0'\in R(\mathfrak{g})$ such that $\Delta({r_0})(r_0')\neq 0$.  Since $\mathfrak{g}$ is perfect,  there exists a $\mathfrak{sl}(2)$-triple $\mathfrak{s}_1$ of $\mathfrak{s}$ such that $\phi(\mathfrak{s}_1)\Delta({r_0})(r_0')\neq \lbrace 0 \rbrace$. Set $N$ the sum of all non-trivial $\mathfrak{s}_1$-submodules of the $\mathfrak{s}_1$-module $R(\mathfrak{g})$ so that $R(\mathfrak{g})=N\oplus M$ where $M=\cap_{a\in \mathfrak{s}_1} \mbox{Ker}(\phi(a))=R(\mathfrak{g})^{\mathfrak{s}_1}$. Let $\pi : R(\mathfrak{g})=N\oplus M \to N$ be the canonical projection,  consider $r_1=\pi(r_0)$, $r_1'=\pi(r_0')$ and choose $a_1\in\s_1$ such that $\phi(a_1)\Delta({r_0})(r_0')\ne 0$.
 Using identities (2) and (4) of Corollary \ref{corol21}, we have
\begin{eqnarray*}
 0&\ne&\phi(a_1)\Delta(r_0)r_0'=\Delta(r_0)\phi(a_1)r_0'=\Delta(r_0)\phi(a_1)r_1'=\phi(a_1)\Delta(r_0)r_1'=\phi(a_1)\Delta(r_1')r_0\\&=&\Delta(r_1')\phi(a_1)r_0=\Delta(r_1')\phi(a_1)r_1=\phi(a_1)\Delta(r_1')r_1=\phi(a_1)\Delta(r_1)r_1'=\phi(a_1)\pi\Delta(r_1)r_1'.
\end{eqnarray*}
Thus, there exist $r_1,r_1'\in N$ such that $\phi(\s_1)\pi\Delta(r_1)r_1\ne\{0\}$; this obviously implies that $\pi\Delta(r_1)r_1'\ne 0$. Define  $\tilde\g=\s_1\ltimes_\varphi N$ where $\varphi=\phi_{|\s_1}$, with $R(\g)=N$ abelian. It is clear that $\tilde\g$ is perfect. Define on $\tilde\g$ the product $(a+r)\bullet (a'+r')=\pi\Delta(r)(r')$ for all $s,s'\in\s_1$, $r,r'\in N$. It is obviously nonzero since $r_1\bullet r_1'\ne 0$ and one easily sees that it verifies the conditions of Corollary \ref{corol21}, and hence defines an $A_{BD}$-structure on $\tilde\g$. But, since $\s_1$ is isomorphic to $\mathfrak{sl}(2)$, this   contradicts Proposition \ref{sl2}. We thus conclude that $\Delta(r_0)= 0.$
	
Since ${r_0}\circ \g\neq \{0\}$ but $\Delta(r_0)= 0$, we must have $Gr_0\ne 0$ and then we may find a ${\mathfrak{sl}}(2)$-triple $\s_1$ in $\s$ such that $\phi(\s_1)Gr_0\ne \{0\}$. Decompose $R(\g)$, as before, into the sum of two $\s_1$-submodules $R(\mathfrak{g})=N\oplus M$ where $M=R(\g)^{\s_1}$ and let $\pi:N\oplus M\to N$ be  the natural projection. Since $\s_1=[\s_1,\s_1]$, the condition $\phi(\s_1)Gr_0\ne \{0\}$ implies that there exist $a_1,a_1'\in\s_1$ such that $\phi([a_1,a_1'])Gr_0\ne \{0\}$. But then either $\phi(a_1)\phi(a_1')Gr_0\ne 0$ or $\phi(a_1')\phi(a_1)Gr_0\ne 0$ because $\phi([a_1,a_1'])=[\phi(a_1),\phi(a_1')]$. We may suppose without loss of generality $\phi(a_1')\phi(a_1)Gr_0\ne 0$. But in such a case, if $r_1=\pi(r_0)$,  using  identity (3) of Corollary \ref{corol21}, we get
$$
 0\ne\phi(a_1')\phi(a_1)Gr_0=\phi(a_1)G\phi(a_1')r_0=\phi(a_1)G\phi(a_1')r_1
=\phi(a_1')\phi(a_1)Gr_1=\phi(a_1')\phi(a_1)\pi(Gr_1).
$$
This clearly implies that there exist $a_1\in\s_1$ and $r_1\in N$ such that $\phi(a_1)\pi(Gr_1)\ne 0$. Now, define on the perfect Lie algebra $\tilde\g=\s_1\ltimes_\varphi N$, where $\varphi=\phi_{|\s_1}$ and $N$ is abelian, the product $(a+r)\bullet (a'+r')=\varphi(a)\pi(Gr')+\varphi(a')\pi(Gr)$, for all $a,a'\in\s_1$, $r,r'\in N$. One easily sees that the conditions of the second case in Remark \ref{rema22} are verified and thus we would obtain a nontrivial $A_{BD}$-structure on $\tilde\g$, which is impossible, according to Proposition \ref{sl2}. As a result,  ${r_0}\circ \g= \{0\}$, showing that $R(\g)\circ \g=\{0\}$.

The consequences are now straightforward from Corollary \ref{corol21}.
\end{proof}
	
\begin{lemma}\label{pr1} Let $\g=\s\ltimes_\phi R(\g)$ be a perfect Lie algebra with abelian radical and let $\circ$ be an $A_{BD}$-structure  on $\g$. For every subalgebra $\s_1$ defined by a ${\mathfrak{sl}}(2)$-triple, it holds that $\s_1\circ\s_1=\{0\}$.
\end{lemma}
\begin{proof}
Suppose that there exist $a,a'\in\s_1$ such that $a\circ a'\ne 0$. According to Proposition \ref{le2}, there exists $F:\s\to R(\g)$ such that $a\circ a'=\phi(a)Fa'$. Let $N$ be the sum of nontrivial $\s_1$-submodules of $R(\g)$,   put $M=R(\g)^{\s_1}$ and let $\pi:R(\g)=N\oplus M\to N$ be the natural projection. On the perfect Lie algebra $\tilde\g=\s\ltimes_{\varphi} N$, where $\varphi=\phi_{|N}$ and  $R(\tilde\g)=N$ is abelian, let us define a product  by $(a+r)\bullet (a'+r')=\varphi(a)\pi F(a').$ It is easy to deduce from Proposition \ref{proposition21} that $\bullet$ provides a nontrivial $A_{BD}$-structure on $\tilde\g$, a contradiction with Proposition \ref{sl2}. As a consequence, $\s_1\circ\s_1=\{0\}$.
\end{proof}

\begin{corollary}\label{corollary23}  Let $\g=\s\ltimes_\phi R(\g)$ be a perfect Lie algebra with abelian radical and let $\circ$ be an $A_{BD}$-structure  on $\g$.
If $\h$ is a  Cartan subalgebra of $\s$, then $\h\circ \mathfrak{g}=\lbrace 0 \rbrace$.
\end{corollary}
\begin{proof} Let us first see that $\h\circ \h=\lbrace 0 \rbrace$. Suppose, on the contrary, that there exist  $h_1,h_2\in \h$ such that $h_1\circ h_2 \neq 0$. For $h_0=h_1+h_2$, we have, $h_0\circ h_0=h_1\circ h_1+h_2\circ h_2+2h_1\circ h_2$. Thus, $h_0\circ h_0-h_1\circ h_1 - h_2\circ h_2=2h_1\circ h_2 \neq 0$.
	Consequently, there exists $j\in\{0,1,2\}$ such that  $h_j\circ h_j \neq 0$. But, since $h_j\in\h$, we can find a  $\mathfrak{sl}(2)$-triple $\s_1$ such that $h_j\in\s_1$ and Lemma \ref{pr1} implies $h_j\circ h_j\in \s_1\circ\s_1=\{0\}$, a contradiction. Hence, we must have  $\h\circ \h=\lbrace 0 \rbrace$.  
	
		Let us now consider $h, h_1 \in \h$, and let $\lbrace h_1,e_1,f_1 \rbrace$ be a $\mathfrak{sl}(2)$-triple of $\mathfrak{s}$ with $[h_1,e_1]=2e_1$, $[h_1,f_1]=-2f_1$, $[e_1,f_1]=h_1$. Let $F:\s\to R(\g)$ be the map such that $a\circ a'=\phi(a)Fa'$, guaranteed by Proposition \ref{le2}. As $\h\circ \h=\lbrace 0 \rbrace$, we obviously have $\phi(h)F(h_1)=\phi(h_1)F(h)=0$, and the identity $\s_1\circ\s_1=\{0\}$ gives $\phi(h_1)F(e_1)=\phi(h_1)F(f_1)=0$. Therefore, using also that $\phi(h)\phi(h_1)=\phi(h_1)\phi(h)$ because $[h,h_1]=0$, we obtain		
		\begin{eqnarray*}
2\phi(e_1)Fh&=&\phi([h_1,e_1])Fh= \phi(h_1)\phi(e_1)Fh- \phi(e_1)\phi(h_1)Fh=
\phi(h_1)\phi(e_1)Fh\\&=&\phi(h_1)\phi(h)Fe_1=\phi(h)\phi(h_1)Fe_1=0,\\
2\phi(f_1)Fh&=&\phi([f_1,h_1])Fh= \phi(f_1)\phi(h_1)Fh- \phi(h_1)\phi(f_1)Fh
=-\phi(h_1)\phi(f_1)Fh\\&=&-\phi(h_1)\phi(h)Ff_1=\phi(h)\phi(h_1)Ff_1=0,\end{eqnarray*}
and one gets that $h\circ x=\phi(x)Fh=0$ for all $x$ in the ${\mathfrak{sl}}(2)$-triple. Now the result follows from the fact that $\s$ admits a basis which is a union of ${\mathfrak{sl}}(2)$-triples.\end{proof}

		\medskip
		
		We can now prove the main theorem of the paper.
		
		\smallskip
	
\begin{theorem}\label{main}
	Any $A_{BD}$-structure on a perfect Lie algebra over an arbitrary field of characteristic zero is trivial.
\end{theorem}
\begin{proof} Let $\g$ be a perfect  Lie algebra over the field $\K$ and suppose that it admits a nontrivial $A_{BD}$-structure. According to Proposition \ref{propreduc}, we may find a perfect Lie algebra $\overline\g=\s\ltimes_\phi R(\overline\g)$, over the algebraic closure $\overline\K$ of $\K$,  with abelian radical also admitting a nontrivial $A_{BD}$-structure $\circ$.

	If $a\in \mathfrak{s}$, then $a=a_s+a_n$ where $a_s$ is  semisimple and $a_n$ is nilpotent. Since, $a_s$ is semisimple then there exists a Cartan subalgebra $\h$ of $\mathfrak{s}$ such that $a_s\in \h$ \cite[Prop. 10,p. 24]{Bour}. By Corollary \ref{corollary23}, we have $a_s\circ \tilde\g=\lbrace 0 \rbrace$. On the other hand, $a_n$ is nilpotent and then, by Jacobson-Morozov Theorem \cite[Prop. 2, p. 162]{Bour}, there exists a $\mathfrak{sl}(2)$-triple of $\mathfrak{s}$ containing $a_n$. Hence, $a_n\circ a_n=0$, by Lemma \ref{pr1}.
	Consequently, $a\circ a=(a_s+a_n)\circ(a_s+a_n)=a_s\circ(a_s+a_n)+a_n\circ a_s +a_n\circ a_n=0$, for all $a\in\s$.
	But now, if we consider $a,a'\in \mathfrak{s}$, then $(a+a')\circ(a+a')=0$, which implies that $a\circ a +2a\circ a'+a'\circ a'=0$ and, therefore, $a\circ a'=0$, so $\mathfrak{s}\circ \mathfrak{s}=\lbrace 0\rbrace$ and Proposition \ref{le2} shows that $\overline\g\circ \overline\g=\lbrace 0 \rbrace$, which contradicts the fact that $\circ$ is nontrivial. 
	
	As a consequence, $\g$ cannot admit a nontrivial $A_{BD}$-structure either. The result for a real Lie algebra follows at once from Lemma \ref{complexif}
\end{proof}

\medskip

We can use the equivalence of Theorem \ref{Equi} to reformulate our Theorem \ref{main} in order to give answer to the open question posed in \cite{br2} as follows:

\smallskip

\begin{theorem}
 Let $\g$ be a perfect   Lie algebra of finite dimension over an arbitrary field of characteristic zero and let $M$ be a $\g$-module of finite dimension.
			Any biderivation of $\g$ with values in $M$ is trivial. 
\end{theorem}

\medskip

\begin{remark}{\em Notice that, since every CPA-structure is actually an $A_{BD}$-structure, 
we rediscover from Theorem \ref{main} the result of D. Burde and W. A. Moens assuring that perfect  Lie algebras do not admit nontrivial CPA-structures \cite[Th. 3.3]{burde2}. }
\end{remark}


\vskip 1 cm

\begin{thebibliography}{99}
	
	\bibitem{BO}
	S. Benayadi and H. Oubba,  \emph{Nonassociative Algebras of Biderivation-type},  Linear Algebra  Appl. 701 (2024) 22--60.
	
	\bibitem{Bour}
	N. Bourbaki,  \emph{Groupes at algèbres de Lie, Chapitres 7 et 8},  Springer-Verlag, Berlin, 2006.
	
	\bibitem{br2}
	M. Bre\v{s}ar, K. Zhao, \emph{Biderivations and  commuting Linear maps on  Lie algebras,} J. Lie Theory 28 (1018),  885--900.
	
	\bibitem{BuDe} D. Burde and K. Dekimpe,
	\emph{Post-Lie algebra structures on pairs of Lie algebras},
	J. Algebra 464 (2016), 226--245
	
	\bibitem{BD1} D. Burde and K. Dekimpe,
	\emph{Post-Lie algebra structures and generalized derivations of semisimple Lie algebras}, Moscow Math. J.,  13:1 (2013), 1--18.
	
	
	
	\bibitem{burde2} D. Burde and W. A. Moens,
	\emph{Commutative post-Lie algebra structures on Lie algebras}, J. Algebra,  467 (2016), 183--201.
	 
	\bibitem{jacob}  N. Jacobson, \textit{Lie algebras}, Dover Pub., New York, 1979.

  
	\bibitem{Serre} J-P. Serre,  \emph{Complex Semisimple Lie Algebras},  Springer-Verlag,  New York, 1987.
	
	
\end{thebibliography}
\end{document}